\documentclass[11pt]{amsart}
\usepackage
[    left=3.5cm,
    right=3.5cm,
]
{geometry}

\usepackage{tikz}
\usetikzlibrary{matrix}
\usepackage[implicit=false]{hyperref}
\hypersetup{colorlinks=true,linkcolor=blue,urlcolor=blue}
\usepackage{lineno}
\usepackage{amssymb}
\usepackage{verbatim}
\usepackage{array}
\usepackage{latexsym}
\usepackage{enumerate}
\usepackage{amsmath}
\usepackage{amsfonts}
\usepackage{amsthm}
\usepackage{color}
\usepackage[english]{babel}

\makeatletter
\usepackage{comment}
\let\wfs@comment@comment\comment
\let\comment\@undefined

\usepackage{changes}
\let\wfs@changes@comment\comment
\let\comment\@undefined

\newcommand\comment{%
    \ifthenelse{\equal{\@currenvir}{comment}}
    {\wfs@comment@comment}
    {\wfs@changes@comment}%
}

\definechangesauthor[name=Korchmaros, color=green]{Korchmaros}
\definechangesauthor[name=Pietro, color=red]{Pietro}
\definechangesauthor[name=Herivelto, color=blue]{Herivelto}

\newtheorem{theorem}{Theorem}[section]

\newtheorem{prop}[theorem]{Proposition}

\newtheorem{lem}[theorem]{Lemma}
\newtheorem{corollary}[theorem]{Corollary}

{\theoremstyle{definition}
\newtheorem{definition}{Definition}}
\newtheorem{rem}[theorem]{Remark}
\newtheorem{question}[theorem]{Question}

\newcommand{\dsum}{\displaystyle\sum}
\newcommand{\VarP}{\textbf{\textsc{v}}}
\def\min{{\rm min}}

\begin{document}
\title{Optimal plane curves of degree $q-1$ over a finite field}

 \thanks{Walteir de Paula Ferreira is with the Instituto de Matem\'atica, Estat\'istica e Computa\c{c}\~ao Cient\'ifica, Universidade Estadual de Campinas, Campinas, SP 13083-859, Brazil.\\ e-mail: w235426@dac.unicamp.br}
 \thanks{Pietro Speziali is with the Instituto de Matem\'atica, Estat\'istica e Computa\c{c}\~ao Cient\'ifica, Universidade Estadual de Campinas, Campinas, SP 13083-859, Brazil.\\ e-mail: speziali@unicamp.br}

\thanks{{\bf Keywords}: Algebraic curves; Rational points; Sziklai bound; optimal curves; arcs}

\thanks{{\bf Mathematics Subject Classification (2010)}: 114H37, 14H05}

\author{Walteir de Paula Ferreira}
\author{Pietro Speziali}
\date{November 2022}

\begin{abstract}  
    Let $q\geq 5$ be a prime power. In this note,  we prove that if a plane curve $\mathcal{X}$ of degree $q - 1$ defined over $\mathbb{F}_q$ without $\mathbb{F}_q$-linear components attains the Sziklai upper bound $(d-1)q+1 = (q - 1)^2$ for the number of its  $\mathbb{F}_q$-rational points, then $\mathcal{X}$ is projectively equivalent over $\mathbb{F}_q$ to the curve $ \mathcal{C}_{(\alpha,\beta,\gamma)} : \alpha X^{q - 1} + \beta Y^{q - 1} + \gamma Z^{q - 1} = 0$ for some $\alpha, \beta, \gamma \in \mathbb{F}_q^{*}$ such that $\alpha + \beta + \gamma = 0$. This completes the classification of curves that are extremal with respect to the Sziklai bound. 
    Also, since the Sziklai bound is equal to the St\"ohr-Voloch's bound for plane curves of degree $q - 1$, our main result classifies the $\mathbb{F}_q$-Frobenius extremal classical plane curves of degree $q - 1$.
\end{abstract}

\maketitle

\section{Introduction} 

Let $\mathcal{X}$ be a (projective, geometrically irreducible, algebraic) curve defined over a finite field $\mathbb{F}_q$ where $q = p^h$ for some prime $p$ and some positive integer $h$. It is a classical problem to count the number $\text{N}_q(\mathcal{X})$  of $\mathbb{F}_q$-rational points of $\mathcal{X}$. However, since this problem is rather hard to solve, it is often desirable to find good upper bounds for  $\text{N}_q(\mathcal{X})$ depending on some invariants of the curve $\mathcal{X}$. For instance, the famous Hasse-Weil upper bound states that $\text{N}_q(\mathcal{X}) \leq 1 + q + 2g\sqrt{q}$ where $g$ is the genus of $\mathcal{X}$. Note that the same bound holds for any curve defined over $\mathbb{F}_q$ and genus $g$. 
Once we have a bound, it is a natural question to see whether such a bound is sharp or not, and then, it is also natural to try and classify the optimal curves, that is, the curves attaining said bound. In the context of the Hasse-Weil bound, such optimal curves do exist and are called maximal curves. Maximal curves may exist when $q = n^2$ is a square, and it is known that the genus of an $\mathbb{F}_q$-maximal curve is upper bounded by $(n-1)(n-2)/2$. Up to birational equivalence, there is exactly one $\mathbb{F}_q$-maximal curve of genus $(n-1)(n-2)/2$: the Hermitian curve $\mathcal{H}_n$ given by the homogeneous equation
$$
\mathcal{H}_n : Y^nZ + YZ^n = X^{n+1}. 
$$
It is a classical and yet unsolved problem to find the spectrum of the genera of $\mathbb{F}_q$-maximal curves; see \cite{Spectrum}.

Now, let $\mathcal{X}$ be a plane curve of degree $d \geq 2$ without $\mathbb{F}_q$-linear components. In \cite{PeterS}, Sziklai conjectured the following result:  
$\text{N}_q(\mathcal{X}) \leq (d - 1)q + 1$. 
The unique exception to Sziklai's conjecture was found in \cite[Section 3]{Homma1} and it is given by the curve over $\mathbb{F}_4$ with  homogeneous equation
$$
X^4 + Y^4 + Z^4 + X^2Y^2 + Y^2Z^2 + Z^2X^2 + X^2YZ + XY^2Z + XYZ^2 = 0.
$$

The Sziklai bound was later proved by Homma and Kim in a sequence of three papers \cite{Homma1, Homma2, Homma3}. We are interested in curves attaining the Sziklai bound.  Let $\mathcal{X}$ be a nonsingular plane curve that is optimal with respect to the Sziklai upper bound;  then  \cite[Section 5]{Homma2} its degree $d$ must belong to the set $\{2, \sqrt{q} + 1, q - 1, q, q + 1, q + 2\}$. This means that the spectrum of the degrees of optimal Sziklai curves is pretty small, hence, it seems feasible to classify, up to projective equivalence, the curves attaining the Sziklai bound. So far, this  is known in the cases $d = 2, \sqrt{q} + 1, q, q + 1$ or $q + 2$; see \cite[Section 5.1]{James}), \cite[Theorem 8]{HermitianC}), \cite[Main Theorem]{Hommaq}, \cite[Theorem 1.3]{Hommaq11}, \cite{Tallini} and \cite{HKTallini}. For the case $d = q - 1$, a family of optimal curves is  given by the homogeneous equation
$$
C_{(\alpha, \beta, \gamma)} : \alpha X^{q - 1} + \beta Y^{q - 1} + \gamma Z^{q - 1} = 0
$$ 
with $\alpha, \beta, \gamma \in \mathbb{F}_q^{*}$ and $\alpha + \beta + \gamma = 0$. Recently, Homma \cite{Hommaq_1} has stated the following question:
\begin{question}\label{q1}
Are there curves of degree $q - 1$ that attain the Sziklai's upper bound such that are not projectively equivalent over $\mathbb{F}_q$ to a curve of type $\mathcal{C}_{(\alpha, \beta, \gamma)}$?
\end{question}

In the same paper, he gives a positive solution to this problem for $q =4$, since in this case, the Hermitian cubic attains Sziklai's bound but is not projectively equivalent to any $\mathcal{C}_{(\alpha, \beta, \gamma)}$.

In this paper, we give a negative answer to Question \ref{q1} for $q\geq 5$, thus completing the classification of optimal Sziklai curves; see Theorem \ref{MainResult}.  

This paper is organized as follows. 

In Section 2, we give the necessary background, as well as, a brief survey of the existing literature on the Sziklai bound and related topics. 

In Section 3, we will give several technical results that are necessary to prove our classification of curves of degree $q-1$ that are optimal with respect to the Sziklai bound.

 Section 4 is devoted to the proof of  Theorem \ref{MainResult}, which is the main result of our paper. Here, we remark that while our technique applies to all $q\geq 8$, the cases $q = 5, 7$ need to be dealt with by using two different approaches, which are of independent interest. The former needs the knowledge of L-polynomial of curves of genus 3 with small defect \cite{Lauter}, the latter is based on the classification on $(36,6)$-arcs in $\mathbb{P}^2(\mathbb{F}_7)$ \cite{ArcsPG27}.

 Finally, in Section 5, we give a brief discussion regarding topics that are directly linked to (or are possible applications of) our results. More in detail, we show as our main result is related to the $\mathbb{F}_q$-Frobenius classical plane curves of degree $q-1$ attaining the Stöhr-Voloch upper bound. 
 Further, curves attaining the Sziklai upper bound are related to nonsingular hypersurfaces with many $\mathbb{F}_q$-rational points in even-dimensional projective spaces; see \cite{Datta2009, Tironi2022}.

\section{Background and preliminary results}

\subsection{Plane Curves}
    Our notation and terminology are standard. For definitions and basic properties of plane curves, see \cite[Chapter 1-5]{Torres}. In the projective plane $\mathbb{P}^2(\mathbb{F}_q)$, a plane curve $\mathcal{C}$ in $\mathbb{P}^2(\mathbb{F}_q)$ of homogeneous equation $F(X, Y, Z) = 0$, where $F \in \mathbb{F}_q[X, Y, Z]$, is denoted by $\mathcal{C} = \VarP(F)$ and consists of all points $(x: y: z) \in \mathbb{P}^2(\mathbb{F}_q)$ such that $F(x, y, z) = 0$. We regard $\mathcal{C}$ as a curve over a fixed algebraic closure $K := \overline{\mathbb{F}}_q$. By doing so, the points in $(x: y: z) \in \mathbb{P}^2(\mathbb{F}_q)$ such that $F(x, y, z) = 0$ are called $\mathbb{F}_q$-rational points (or simply, rational points) of $\mathcal{C}$ and $\mathcal{C}(\mathbb{F}_q) := \mathbb{P}^2(\mathbb{F}_q) \cap \mathcal{C}$. A plane curve $\mathcal{C} = \VarP(F)$ in $\mathbb{P}^2 := \mathbb{P}^2(K)$, where $F \in K[X, Y, Z]$, is said to be defined over $\mathbb{F}_q$ if there is $\lambda \in K^{*}$ such that $\lambda \cdot F \in \mathbb{F}_q[X, Y, Z]$. A point $P = (x: y: z)$ of a plane curve $\mathcal{C}$ in $\mathbb{P}^2$ is singular if
    $$\dfrac{\partial F}{\partial X}(x, y, z) = \dfrac{\partial F}{\partial Y}(x, y, z) = \dfrac{\partial F}{\partial Z}(x, y, z) = 0.$$
    Otherwise, $P$ is a nonsingular (or a simple) point. The degree of $\mathcal{C}$, denoted by $\text{deg}(\mathcal{C})$,  is $\text{deg}(F)$. A component of $\mathcal{C}$ is a curve $\mathcal{X}$ of homogeneous equation $G = 0$ such that $G|F$. A curve $\mathcal{C}$ is said to be absolutely irreducible if $F$ is irreducible over $K$. A projectivity $\varphi_A: \mathbb{P}^2 \rightarrow \mathbb{P}^2$ is given as follows:
    $$\varphi_A(x: y: z) = A \cdot \left(\begin{array}{c}
         x \\
         y \\
         z
    \end{array}\right)$$
    where $A \in \text{GL}_3(K)$. Two curves $\mathcal{C}$ and $\mathcal{X}$ are said projectively equivalent over $\mathbb{F}_q$, denoted by $\mathcal{C} \simeq_{\text{proj}} \mathcal{X}$, if there is a projectivity $\varphi_A: \mathbb{P}^2 \rightarrow \mathbb{P}^2$ with $A \in \text{GL}_3(\mathbb{F}_q)$ such that $\varphi_A(\mathcal{C}) = \mathcal{X}$. The $q$-Frobenius map $\Psi_q:\mathbb{P}^2 \rightarrow \mathbb{P}^2$ is defined by $\Psi_q(x: y: z) := (x^q: y^q: z^q)$. The dual projective space $\check{\mathbb{P}}^2 = \check{\mathbb{P}}^2(K)$ is the space of all line in $\mathbb{P}^2$ and $\check{\mathbb{P}}^2(\mathbb{F}_q)$ we mean the set of lines defined over $\mathbb{F}_q$ of $\check{\mathbb{P}}^2.$ For a point $P \in \mathbb{P}^2(\mathbb{F}_q),$ define $\check{P}(\mathbb{F}_q) := \{ l \in \check{\mathbb{P}}^2(\mathbb{F}_q)\ |\ P \in l\}.$ 
    
    The proof of the next result can be found in \cite[section 4.5]{Torres}.
    \begin{theorem}[Nother’s “AF + BG” Theorem]\label{Noether}
        Let $\mathcal{F} = \VarP(F)$ and $\mathcal{G} = \VarP(G)$ be two plane curves defined over $\mathbb{F}_q$ with no common components. If $\mathcal{F} \cap \mathcal{G} = \{P_1, ..., P_s\}$ and the multiplicity of $\mathcal{F}$ and $\mathcal{G}$ at each point $P_i$ is equal to $1$, then for all plane curve $\mathcal{X} = \VarP(H)$ defined over $\mathbb{F}_q$ with $\mathcal{F} \cap \mathcal{G} \subseteq \mathcal{X}$ there are $A, B \in \mathbb{F}_q[X, Y, Z]$ such that $H = AF + BG$.
    \end{theorem}
    
    Note that Theorem \ref{Noether} gives a method to find all plane curves passing through a given set of points of $\mathbb{P}^2(\mathbb{F}_q)$.

\subsection{Arcs of $\mathbb{P}^2(\mathbb{F}_q)$ and Codes} 

    The following brief account of the theory of  plane arcs and their relationship to linear codes is based on  \cite[Chapter 10 and 17]{CodTheo}, to which we refer the reader for further details. In the projective plane $\mathbb{P}^2(\mathbb{F}_q)$, a $(k, n)$-arc $\mathcal{K}$ is a set of $k$ points such that each line contains at most $n$ point of $\mathcal{K}$ and there is a line that contains exactly $n$ points of $\mathcal{K}$. A $(k, 2)$-arc is simply called an arc. Let $\mathcal{K} \subseteq \mathbb{P}^2(\mathbb{F}_q)$ be a $(k, n)$-arc, for $0 \leq i \leq q + 1$, define $\mathcal{A}_i(\mathcal{K}) := \{ l \in \check{\mathbb{P}}^2(\mathbb{F}_q)\ |\ \#(l \cap \mathcal{K}) = i\}$, $a_i(\mathcal{K}) := \#\mathcal{A}_i(\mathcal{K})$ and $k_0(\mathcal{K}) := \min\{ i\ |\ a_i \neq 0\}$. When there is no possibility of confusion we will denote it simply by $\mathcal{A}_i$, $a_i$ and $k_0$.

    A linear subspace $C$ of $\mathbb{F}_q^n$ %(as a vector space over $\mathbb{F}_q$) 
    of dimension $k$ is called an $[n, k]_q$-code. The elements of a linear code $C$ are called codewords. The weight of a codeword $x = (x_1, ..., x_n) \in C \subseteq \mathbb{F}_q^n$ is the number of nonzero coordinates in $x$, denoted by $wt(x)$. The minimum distance of $\mathcal{C}$ is $\min\{ wt(x) | x \in \mathcal{C}, x \neq 0\}$. If the minimum distance of $\mathcal{C}$ is $d$, then we write that  $\mathcal{C}$ is an $[n, k, d]_q$-code. A generator matrix $G$ of an $[n, k, d]_q$-code $\mathcal{C}$ is a matrix with $k$ rows and $n$ columns whose rows form a basis of $\mathcal{C}$. The code $\mathcal{C}$ is recovered from $G$ by taking all linear combinations of rows. If $\mathcal{C}$ contains $c_i$ codewords of weight $i$, for $i = 1, ..., n$, then the weight enumerator is defined by $$W_{\mathcal{C}}(z) := c_0 + c_1z + c_2z^2 + \cdots + c_nz^n \in \mathbb{Z}[z].$$
    
    Let $\mathcal{C}$ be a linear $[n, 3, d]_q$-code described by a generator matrix $G$. We assume that there is no $0$ column in $G$. We can then consider the columns of $G$ as generators of points in $\mathbb{P}^2(\mathbb{F}_q)$. A linear $[n, 3, d]_q$-code $\mathcal{C}$ is called projective if there is a generator matrix whose columns generate different points in $\mathbb{P}^2(\mathbb{F}_q)$. For a projective $[n, 3, d]_q$-code $\mathcal{C}$ with a generator matrix $G$, the $n$ points in $\mathbb{P}^2(\mathbb{F}_q)$ corresponding to columns of $G$ form an $(n, n - d)$-arc in $\mathbb{P}^2(\mathbb{F}_q)$. For each $i$ in $0,\ldots,n-d$, the number $a_i$ of lines in $\mathcal{A}_i$ is related to the coefficients $c_i$ of the weight enumerator as follows: $(q - 1) \cdot (a_0, \cdots, a_{n - d}) = (c_n, \cdots, c_{d})$. %and vice versa.% see \cite[Chapter 10 and 17]{CodTheo}.
\subsection{St\"ohr-Voloch Theory}
In \cite{SVoloch}, Stöhr and Voloch gave a geometric method to obtain upper bounds for the number of rational points of a curve of $\mathbb{P}^n$. Here, we give the necessary background on a particular case of the St\"ohr-Voloch theorem that we will need in Section 4.  

Let $\mathcal{C}$ a plane curve defined over $\mathbb{F}_q$ and $0 = \epsilon_0 < \epsilon_1 = 1 < \epsilon_2$ the order sequence of $\mathcal{C}$. Let $P \in \mathcal{C}$. If $\mathcal{C}$ is nonsingular, then the hermitian $P$-invariants are $j_0(P) = 0, j_1(P) = 1$ and $j_2(P) = \text{I}(P, \mathbb{T}_P(\mathcal{C}) \cap \mathcal{C})$ where $\mathbb{T}_P(\mathcal{C})$ is the tangent line to $\mathcal{C}$ at $P$ and $\text{I}(P, \mathbb{T}_P(\mathcal{C}) \cap \mathcal{C})$ is the intersection multiplicity of $\mathbb{T}_P(\mathcal{C})$ and $\mathcal{C}$ at $P$. Since $\mathcal{C}$ is defined over $\mathbb{F}_q$, then there is a smallest integer $\nu \in \{1, \epsilon_2 \}$ such that
$$\text{det}\left(\begin{array}{ccc}
	x_0^q & x_1^q & x_2^q \\
        x_0 & x_1 & x_2 \\
	D_{\zeta}^{(\nu)} x_0 & D_{\zeta}^{(\nu)} x_1 & D_{\zeta}^{(\nu)} x_2 \\
\end{array}\right) \not\equiv 0$$
where $D_{\zeta}^{(k)}$ is the $k$-th Hasse derivative with respect to a separating variable $\zeta$ of $K(\mathcal{C})|K$, and $x_0, x_1, x_2$ are the coordinate functions on $\mathcal{C} \subseteq \mathbb{P}^2$. The number $\nu$ is called the $q$-Frobenius order of $\mathcal{C}$, and such a curve is called $q$-Frobenius classical if $\nu = 1$. 
%Aqui faltou o Teorema "0" De S-V.
The following is a slight rewording of \cite[Theorem 0.1]{SVoloch}.
\begin{theorem}\label{babysv}
Let $\mathcal{C}$ be an irreducible plane curve of degree $d$ defined over a finite field $\mathbb{F}_q$. If $\mathcal{C}$ is $q$-Frobenius classical, then 
$$
{ \rm N}_q(\mathcal{C}) \leq \frac{1}{2}d(d+q-1). 
$$
\end{theorem}

A refined version of this theorem can be obtained if one can gather sufficient information on the number and the weight of $\mathbb{F}_q$-rational inflection points. 
\begin{theorem}\label{SVolochTheo}
    Let $\mathcal{C} \subseteq \mathbb{P}^2$ be an irreducible nonsingular algebraic curve of genus $g$ and degree $d$ defined over $\mathbb{F}_q$. If $\nu$ is the $q$-Frobenius order of $\mathcal{C}$, then
    $$2 \cdot { \rm N}_q(\mathcal{C}) \leq \nu(2g - 2) + (q + 2)d - \sum_{P \in \mathcal{C}} A(P)$$
    where $A(P) = j_2(P) - \nu - 1$ if $P \in \mathcal{C}(\mathbb{F}_q)$ and $A(P) = 0$ otherwise.
\end{theorem}
\subsection{Sziklai's upper bound and optimal curves}
    Let $C_d(\mathbb{F}_q)$ be the set of plane curves of degree $d \geq 2$ defined over $\mathbb{F}_q$ without $\mathbb{F}_q$-linear components. For $\mathcal{C} \in C_d(\mathbb{F}_q)$, let $\text{N}_q(\mathcal{C}) := |\mathcal{C}(\mathbb{F}_q)|$. In \cite{PeterS}, Sziklai conjectured the bound $\text{N}_q(\mathcal{C}) \leq (d - 1)q + 1$. Actually, Sziklai's conjecture fails for curves of degree $4$ over $\mathbb{F}_4$, as the plane curve with equation
    \begin{align}\label{Curve14F4}
        X^4 + Y^4 + Z^4 + X^2Y^2 + Y^2Z^2 + Z^2X^2 + X^2YZ + XY^2Z + XYZ^2  & = 0
    \end{align}
	has $14$ points over $\mathbb{F}_4$ (see \cite[section 3]{Homma1}), while Sziklai's bound is equal to $13$. Later on, in a sequence of three papers \cite{Homma1, Homma2, Homma3}, Homma and Kim proved the modified Sziklai’s Conjecture:
	\begin{theorem}[Sziklai's upper bound]
        If $\mathcal{C} \in \mathcal{C}_d(\mathbb{F}_q),$ then
		\begin{align}\label{SziklaiBound}
		    {\rm N}_q(\mathcal{C}) & \leq (d - 1)q + 1,
		\end{align}
		except for the curve over $\mathbb{F}_4$ which is projectively equivalent to curve defined by (\ref{Curve14F4}).
    \end{theorem}
	\begin{rem}\label{RemarkSing}
	   Let $\mathcal{C} \in C_q(\mathbb{F}_q)$ with $(d, q) \neq (4, 4)$. If $N_q(\mathcal{C}) = (q - 1)q + 1$, then $\mathcal{C}$ is absolutely irreducible and any rational point of $\mathcal{C}$ is nonsingular, see \cite[section 2]{Homma2}.     
	\end{rem}
	
	In  \cite[section 5]{Homma2}, Homma and Kim observe that the possible degrees $d$ of a nonsingular curve with $(d - 1)q + 1$ rational points are $q + 2, q + 1, q, q - 1, \sqrt{q} + 1$ and $2$. Also, for each degree $d$ in the list, there is a nonsingular curve of degree $d$ attaining the bound. For $d \neq q-1$, the complete classification of such optimal curves is known; we summarize these results in the following Theorem. 
	
	%The known classification, up to projective equivalence, is presented in the following theorem:
	\begin{theorem}
	    Let $\mathcal{C} \in \text{C}_d(\mathbb{F}_q)$ a nonsingular curve with ${\rm N}_q(\mathcal{C}) = (d - 1)q + 1$.
	        \begin{itemize}
		        \item[(i)] If $d = 2$, then $\mathcal{C} \simeq_{{\rm proj}} \VarP(X^2 + YZ)$ over $\mathbb{F}_q$ {\rm (\cite[Section 5.1]{James})}. 
		        \item[(ii)] If $d = \sqrt{q} + 1$, then $\mathcal{C} \simeq_{{\rm proj}} \VarP(X^{\sqrt{q} + 1} + Y^{\sqrt{q} + 1} + Z^{\sqrt{q} + 1})$ over $\mathbb{F}_q$ when $q > 4$ is a square {\rm (\cite{HermitianC})}.
		        \item[(iii)] If $d = q + 2$, then $\mathcal{C}$ is projectively equivalent over $\mathbb{F}_q$ to the curve of type $\VarP(Y(Y^qZ - YZ^q) + Z(Z^qX - ZX^q) + (aX + bY + cZ)(X^qY - XY^q))$
		        where $t^3 - (ct^q + bt + a)$ is irreducible over $\mathbb{F}_q$ {\rm  (\cite{Tallini} and \cite{HKTallini})}.
		        \item[(iv)] If $d = q + 1$, then $\mathcal{C}$ is projectively equivalent over $\mathbb{F}_q$ to the curve  $$\mathcal{C}_{q + 1} := \VarP(X^{q + 1} - X^2Z^{q - 1} + Y^qZ - YZ^q)$$
		        when $q \geq 5$ or $q = 2$. If $q = 4$, then $\mathcal{C}$ is projectively equivalent over $\mathbb{F}_4$ to either $\mathcal{C}_{5}$ or the curve
		        $$\VarP(\mu G(X, Y, Z) + XYZ(\mu^2(X^2 + Y^2 + Z^2) + XY + YZ + ZX))$$
		        where $G(X, Y, Z) = X^4Y + XY^4 + Y^4Z + YZ^4 + Z^4X + ZX^4$ and $\mu^2 + \mu + 1 = 0$. Moreover, those two curves are not projectively equivalent  to each other over $\mathbb{F}_4$. If $q = 3$, then $\mathcal{C}$ is projectively equivalent over $\mathbb{F}_4$  either to $\mathcal{C}_{4}$ or to the curve
		        $$\VarP(X^3Y - XY^3 + Y^3Z - YZ^3 + Z^3X - ZX^3 + XYZ(X + Y - Z)).$$
		        Moreover, those two curves are not projectively equivalent to each other over $\mathbb{F}_3$ {\rm (\cite{Hommaq11})} .
		        \item[(v)] If $d = q$, then $\mathcal{C} \simeq_{{\rm proj}} \VarP(X^q - XZ^{q - 1} + Y^{q - 1}Z - Z^q)$ over $\mathbb{F}_q$ \rm {(\cite{Hommaq})}. 
		  \end{itemize}
	\end{theorem}
	For $d = q - 1$, as was mentioned by Sziklai in \cite{PeterS}, the curve
        $$\mathcal{C}_{(\alpha, \beta, \gamma)} := \VarP(\alpha X^{q - 1} + \beta Y^{q - 1} + \gamma Z^{q - 1})$$
        with $\alpha, \beta, \gamma \in \mathbb{F}_q^{*}$ and $\alpha + \beta + \gamma = 0$ has $(q - 1)^2$ rational points. This curve is nonsingular and the set of its $\mathbb{F}_q$-rational points is
	    $$\mathcal{C}_{(\alpha, \beta, \gamma)}(\mathbb{F}_q) = \mathbb{P}^2(\mathbb{F}_q) \backslash \{X = 0\} \cup \{ Y = 0\} \cup \{ Z = 0\}.$$
	    Recently, Homma \cite{Hommaq_1}, has studied the number of projective equivalence classes over $\mathbb{F}_q$ in this family of curves. More precisely, he proves the following Theorem.
	    
	    \begin{theorem}\cite[Theorem 1.3]{Hommaq_1}
	        The number $\nu_q$ of projective equivalence classes over $\mathbb{F}_q$ in the family of curves $\{ \mathcal{C}_{(\alpha, \beta, \gamma)}\ |\ \alpha, \beta, \gamma \in \mathbb{F}_q^{*},\ \alpha + \beta + \gamma = 0\}$
	        is as follows:
	        \begin{itemize}
	            \item[(i)] Suppose that the characteristic of $\mathbb{F}_q$ is neither $2$ or $3$.
	                \begin{itemize}
	                    \item[(1)] If $q \equiv 2 \mod 3$, then $\nu_q = (q + 1)/6$.
	                    \item[(2)] If $q \equiv 1 \mod 3$, then $\nu_q = (q + 5)/6$.
	                \end{itemize}
	            \item[(ii)] Suppose that $q$ is a power of $3$. Then $\nu_q = (q + 3)/6$.
	            \item[(iii)] Suppose that $q$ is a power of $2$:
	                \begin{itemize}
	                    \item[(1)] If $q \equiv 2 \mod 3$, then $\nu_q = (q - 2)/6$.
	                    \item[(2)] If $q \equiv 1 \mod 3$, then $\nu_q = (q + 2)/6$.
	                \end{itemize}
	        \end{itemize}
	    \end{theorem}
	    
	    In the same paper, the curves of degree $3$ over $\mathbb{F}_4$ are classified.
	    
	    \begin{theorem}\cite[Theorem 3.1]{Hommaq_1}\label{TheoremCaseQ4}
	        Let $\mathcal{C}$ be a nonsingular plane curve of degree $3$ over $\mathbb{F}_4$. If $\text{N}_4(\mathcal{C}) = 9$, then $\mathcal{C}$ is either
	        \begin{itemize}
	            \item[(i)] the Hermitian cubic $\mathcal{H}_3$ given by $\VarP(X^3+Y^3+Z^3)$ or
	            \item[(ii)] projectively equivalent to the curve $\mathcal{C}_{\alpha}$ given by $\VarP(X^3 + \alpha Y^3 + \alpha^2 Z^3)$ where $\mathbb{F}_4 = \{0, 1, \alpha, \alpha^2\}$.
	        \end{itemize}
	    \end{theorem}
	    \begin{rem}
	        It can be shown that the Hermitian cubic $\mathcal{H}_3$ and the curve $\mathcal{C}_{\alpha}$ are birationally equivalent over $\mathbb{F}_4$. Also, they are projectively equivalent over $\mathbb{F}_{2^6}$ {\rm (see, \cite[section 4]{Hommaq_1})}.
	    \end{rem}

\section{Preliminary results}

    In this section, we give several technical results that are necessary to prove our main result Theorem \ref{MainResult}. First of all, not that, by Theorem \ref{TheoremCaseQ4}, we may assume $q \geq 5$. We begin by proving the following Proposition. 
    \begin{prop}\cite[Proposition 2.1]{Hommaq_1}\label{Z(X)XYZ}
        Let $\mathcal{C}$ be a (possibly reducible) plane curve over $\mathbb{F}_q$ of degree $q - 1$. Then $\mathcal{C} \in \{ \mathcal{C}_{(\alpha, \beta, \gamma)}\ |\ \alpha, \beta, \gamma \in \mathbb{F}_q^{*},\ \alpha + \beta + \gamma = 0\}$ if and only if $$\mathcal{C}(\mathbb{F}_q) = \mathbb{P}^2(\mathbb{F}_q) \backslash (\VarP(X) \cup \VarP(Y) \cup \VarP(Z)).$$
    \end{prop}
    Fix a curve $\mathcal{X} \in C_{q - 1}(\mathbb{F}_q)$ with $\text{N}_q(\mathcal{X}) = (q - 1)^2$. Let $Z(\mathcal{X}) := \mathbb{P}^2(\mathbb{F}_q) \backslash \mathcal{X}(\mathbb{F}_q)$. By Proposition \ref{Z(X)XYZ}, if $Z(\mathcal{X}) = (\VarP(X) \cup \VarP(Y) \cup \VarP(Z))(\mathbb{F}_q)$ then $\mathcal{X}= \mathcal{C}_{(\alpha, \beta, \gamma)}$ for some $\alpha, \beta, \gamma \in \mathbb{F}_q^{\times}$ such that $\alpha + \beta + \gamma = 0$. Since the general projective linear  group ${\rm PGL}(2,q)$ $3$-transitively on the set of lines of $\mathbb{P}^2(\mathbb{F}_q)$ and $\#Z(\mathcal{X}) = 3q$, if there are three lines $l_1, l_2, l_3 \in \check{\mathbb{P}}^2(\mathbb{F}_q)$ such that $Z(\mathcal{X}) = (l_1 \cup l_2 \cup l_3)(\mathbb{F}_q)$, then $l_1, l_2, l_3$ are not concurrent and we can choose coordinates $X, Y, Z$ of $\mathbb{P}^2$  such that $l_1 = \VarP(X), l_2 = \VarP(Y)$ e $l_{3} = \VarP(Z)$. This means that, in order to prove our main result, it is enough to show the existence of such three lines. We start by proving the following lemma. 
    
    \begin{lem}\label{LemaArc}
        The set $\mathcal{X}(\mathbb{F}_q) \subseteq \mathbb{P}^2(\mathbb{F}_q)$ is a $((q - 1)^2, q - 1)$-arc.
    \end{lem}
    \begin{proof}
        Since $\text{deg}(\mathcal{X}) = q - 1$, then $\#(l \cap \mathcal{X}(\mathbb{F}_q)) \leq q - 1$ for every line $l \in \check{\mathbb{P}}^2(\mathbb{F}_q)$. Let 
        $t := \max\{ \#(l \cap \mathcal{X}(\mathbb{F}_q))\ |\ l \in \check{\mathbb{P}}^2(\mathbb{F}_q)  \} \leq q - 1.$
        If $P \in \mathcal{X}(\mathbb{F}_q)$ then each line in $\check{P}(\mathbb{F}_q)$ contains at most $t$ points of $\mathcal{X}(\mathbb{F}_q)$. Since $\#\check{P}(\mathbb{F}_q) = q + 1$ then $1 + (q + 1)(t - 1) \geq (q - 1)^2$. Hence, 
        $$q - 1 \geq t \geq \dfrac{q(q - 2)}{q + 1} + 1 = q - 2 + \dfrac{3}{q + 1} > q - 2.$$
        This implies that $t = q - 1$. Therefore, $\mathcal{X}(\mathbb{F}_q)$ is a $((q - 1)^2, q - 1)$-arc in $\mathbb{P}^2(\mathbb{F}_q)$.
    \end{proof}
    For $0 \leq i \leq q + 1$, recall the definition of
    $$\mathcal{A}_i = \{ l \in \check{\mathbb{P}}^2(\mathbb{F}_q)\ |\ \#(l \cap \mathcal{X}(\mathbb{F}_q)) = i\} \text{ and } a_i = \#\mathcal{A}_i.$$
    Since $\text{deg}(\mathcal{X}) = q - 1$, then $a_q = a_{q + 1} = 0$. A line $l \in \check{\mathbb{P}}^2(\mathbb{F}_q)$ is called an $i$-line if $l \in \mathcal{A}_i$. A point $P \in \mathbb{P}(\mathbb{F}_q)$ is said to be of type $i_1^{r_1} ... i_t^{r_t}$ ($i_1 > \cdots > i_t$ and $r_1, ..., r_t \geq 0$) if the number of $i_j$-lines through $P$ is $r_j$ for $j = 1, ..., t$. Also, as $\mathcal{X}(\mathbb{F}_q)$ is a $((q - 1)^2, q - 1)$-arc, we may use the following result:
    \begin{lem}\cite[Lemma 12.1.1]{James}\label{PropA_i}
    With the same notation as above, we have the following equalities. 
        \begin{itemize}
        		\item[(i)] $\dsum_{i = 0}^{q - 1} a_i = q^2 + q + 1.$
        		\item[(ii)] $\dsum_{i = 1}^{q - 1} ia_i = (q + 1)(q - 1)^2.$
        		\item[(iii)] $\dsum_{i = 2}^{q - 1} i(i - 1) a_i = q(q - 2)(q - 1)^2.$
        \end{itemize}
    \end{lem}
    \begin{lem}\label{LemmaType}
        Let $P \in \mathbb{P}^2(\mathbb{F}_q)$ be a point of type $i_1^{r_1} ... i_t^{r_t}$. Then $r_1 + \cdots + r_t = q + 1$. Moreover, 
    	\begin{itemize}
    	   \item[(i)] If $P \in \mathcal{X}$, then $i_j \geq 1$ for all $j = 1, ..., t$ and $1 + \sum r_j(i_j - 1) = (q - 1)^2$.
    	   \item[(ii)] If $P \notin \mathcal{X}$ then $\sum r_ji_j = (q - 1)^2$.
        \end{itemize}
    \end{lem}
    \begin{proof} 
        Since the $\mathbb{F}_q$-lines through $P$ cover the whole plane $\mathbb{P}^2(\mathbb{F}_q)$ and $\text{N}_q(\mathcal{X}) = (q - 1)^2$, the proof is straightforward. 
    \end{proof}
    \begin{corollary}\label{Lema3}
        Let $i$ and $j$ be (not necessarily distinct) non-negative integers. Suppose that there are different $\mathbb{F}_q$-lines $l_1,l_2$ with $l_1 \in \mathcal{A}_i$ and $l_2 \in \mathcal{A}_j$. If $P = l_1 \cap l_2 \in \mathcal{X}(\mathbb{F}_q)$, then $i + j \geq q$.
    \end{corollary}
    \begin{proof}
        Suppose that $P$ is of type $i_1^{r_1} ... i_t^{r_t}$. By Lemma \ref{LemmaType} item (i), we have
    	\begin{align*}
    	    (q - 1)^2 & = 1 + \sum r_j(i_j - 1) \\
    	    & \leq 1 + (i - 1) + (j - 1) + (q - 1)(q - 2) \\
    	    & = i + j - 1 + q^2 - 3q + 2\\
    	    & = i + j - q + (q - 1)^2.
        \end{align*}
    	Therefore, $i + j \geq q$.
    \end{proof}
    
    \begin{definition}
    For $i = 0,\ldots, q-1$ we define $\psi_i: \mathbb{P}^2(\mathbb{F}_q) \rightarrow \{0, 1, ..., q + 1\}$ as  
    $$\psi_i(P) := \#(\check{P}(\mathbb{F}_q) \cap \mathcal{A}_i).$$
    
    \end{definition}
    
    \begin{lem}\label{LemmaPR3}
        If $P \in \mathcal{X}(\mathbb{F}_q)$, then $\psi_{q - 1}(P) \geq 3$. In particular, $a_{q - 1} \geq 3(q - 1)$. Also, if $\psi_{q - 1}(P) = 3$, then $P$ is of type $(q - 1)^3(q - 2)^{q - 2}$. 
    \end{lem}
    \begin{proof}
        Let $r_P = \psi_{q - 1}(P)$. By Lemma \ref{LemmaType} item (i), if $P$ is of type $i_1^{r_1} ... i_t^{r_t}$ then
        \begin{align*}
            (q - 1)^2 & = 1 + \sum r_j(i_j - 1) \\
            & \leq 1 + r_P(q - 2) + (q + 1 - r_P)(q - 3) \\
            & = 1 + qr_P - 2r_P + q^2 - 3q + q - 3 - qr_P + 3r_P \\
            & = r_P + (q - 1)^2 - 3,
        \end{align*}
        hence, $r_P \geq 3$. We get $3(q - 1)^2$ lines in $\mathcal{A}_{q - 1}$. However, each line was counted at most $(q - 1)$ times. This implies $a_{q - 1} \geq 3(q - 1)$. If $\psi_{q - 1}(P) = 3$, let $s_P = \psi_{q - 2}(P)$, then
        \begin{align*}
            (q - 1)^2 & = 1 + \sum r_j(i_j - 1) \\
            & \leq 1 + 3(q - 2) + s_P(q - 3) + (q - 2 - s_P)(q - 4) \\
            & = 3q - 5 + qs_P - 3s_P + q^2 - 4q - 2q + 8 - qs_P + 4s_P \\
            & = s_P + (q - 1)^2 + 2 - q,
        \end{align*}
        hence, $s_P \geq q - 2$. This means that the other lines in $\check{P}(\mathbb{F}_q)$ are in $\mathcal{A}_{q - 2}$. Therefore, $P$ is of type $(q - 1)^3(q - 2)^{q - 2}$. 
        \end{proof}
        
   Loosely speaking, in order to prove our result, we need that there exists a point $Q_0 \in Z(\mathcal{X})$ such that $\psi_{q-1}(Q_0)$ is \emph{big enough}. In order to do so, we give the following proposition, which is inspired by \cite[Proposition 3.1 and Proposition 3.2]{Homma2}.

    \begin{prop}\label{PropQ}
        Fix a point $Q_0 \in Z(\mathcal{X})$ and $l_{\infty} \in \check{\mathbb{P}}^2(\mathbb{F}_q) \backslash \check{Q}_0(\mathbb{F}_q)$. Suppose there are lines $l_1, ..., l_t \in \check{Q}_0(\mathbb{F}_q)\ (2 \leq t \leq q - 1)$ such that $l_i(\mathbb{F}_q) \backslash (\{Q_0\} \cup l_{\infty}) \subseteq \mathcal{X}(\mathbb{F}_q).$
        For a line $l \in \check{Q}_0(\mathbb{F}_q)$ other than these $t$ lines, if  $\#(l\backslash l_{\infty}) \cap \mathcal{X}(\mathbb{F}_q) \geq q - t$, then $l(\mathbb{F}_q) \backslash (\{Q_0\} \cup l_{\infty}(\mathbb{F}_q)) \subseteq \mathcal{X}(\mathbb{F}_q).$
    \end{prop}
    \begin{proof}
        Choose coordinates $X, Y, Z$ of $\mathbb{P}^2$  such that $l_1 = \VarP(X), l_2 = \VarP(Y)$ and $l_{\infty} = \VarP(Z)$, whence $Q_0 = (0: 0: 1)$. Let
        $G_0 = Z^{q - 1} - X^{q - 1} - Y^{q - 1}, G_1 = XY \in \mathbb{F}_q[X, Y, Z]$. Note that $\VarP(G_0)$ and $\VarP(G_1)$ are plane curves with no common components. A direct computation shows that $\#(\VarP(G_0) \cap \VarP(G_1)) = 2(q - 1)$. Also, since $l_i(\mathbb{F}_q) \backslash (\{Q_0\} \cup l_{\infty}) \subseteq \mathcal{X}(\mathbb{F}_q)$ for $i = 1, 2$ then $\VarP(G_0) \cap \VarP(G_1) \subseteq \mathcal{X}(\mathbb{F}_q)$. Let $F$ be a homogeneous equation for $\mathcal{X}$ over $\mathbb{F}_q$, by Nother’s “AF + BG” Theorem \ref{Noether}, we can write
        $$F(X, Y, Z) = a_{00}(Z^{q - 1} - X^{q - 1} - Y^{q - 1}) + XY(g_{q - 3}(X, Y) + g_{q - 4}(X, Y)Z + \cdots + g_0 Z^{q - 3})$$
        where $g_\nu \in \mathbb{F}_q[X, Y, Z]$ is homogeneous of degree $\nu$ and $a_{00} \in \mathbb{F}_q^{*}$. In general, any line $L \in \check{Q}_0(\mathbb{F}_q) \backslash (l_1 \cup l_2)$ is defined by an equation of the form $Y - \mu X = 0$ for some $\mu \in \mathbb{F}_q^{*}.$ So 
        $$L(\mathbb{F}_q) \backslash (Q_0 \cup l_{\infty}(\mathbb{F}_q)) = \{ (1: \mu: \beta)\ |\ \beta \in \mathbb{F}_q^{*} \}.$$ 
        Since $a_{00}(\beta^{q - 1} - 1 - \mu^{q - 1}) = - a_{00}$ when $\beta, \mu \in \mathbb{F}_q^{\times}$, then
        \begin{equation}\label{EqF1}
            F(1, \mu, \beta)  = (\mu(g_{q - 3}(1, \mu) -a_{00}) + \mu q_{q - 4}(1, \mu)\beta + \cdots + \mu g_0 \beta^{q - 3}.
        \end{equation}
        In particular, if $l_{2 + \mu} = \VarP(Y - a_{\mu} X)$ with $\mu = 1, ..., t - 2$, we must have $a_{\mu} \neq 0$. Let
        $$B = \left(\begin{array}{ccccc}
        	&& \vdots&& \\
        	1 & \beta & \beta^2 & \cdots & \beta^{q - 3} \\
        	&& \vdots &&
        \end{array}  \right)_{\beta \in \mathbb{F}_q^* \backslash \{1\}}.$$
        Since $l_{2 + \mu}(\mathbb{F}_q) \backslash (\{Q_0\} \cup l_{\infty}) \subseteq \mathcal{X}(\mathbb{F}_q)$, by equation (\ref{EqF1}), then
        $$B \cdot \left(\begin{array}{c}
        		a_{\mu}g_{q - 3}(1, a_\mu) - a_{00} \\
        		a_{\mu}g_{q - 4}(1, a_\mu) \\
        		\vdots \\
        		a_{\mu} g_0
        	\end{array}\right) = \left(\begin{array}{c}
        		0 \\
        		0 \\
        		\vdots \\
        		0
        	\end{array}\right).$$
        	Since $B$ is a Vandermonde matrix, we have $\text{det}(B) \neq 0$. This implies that 
        	$$g_{q - 4}(1, a_{\mu}) = \cdots = g_0(1, a_{\mu}) = 0.$$
        	If $\nu < t - 2$, since $g_\nu(1, y)$ has $t - 2$ roots $\{ a_1, ..., a_{t - 2}\}$ but its degree is less than $t - 2$, then $g_\nu(1, y) \equiv 0$ as      a polynomial in $y$. So
        	$$F(1, y, z) = a_{00}(z^{q - 1} - y^{q - 1}) + (y g_{q - 3}(1, y) - a_{00}) + y q_{q - 4}(1, y)z + \cdots + y g_{t - 2}(1, y)z^{q - t - 1}.$$
        	Let $l = \VarP(Y - \mu X)$, where $\mu \in \mathbb{F}_q^{*}$, and $\{ (1, \mu, \beta_i)\ |\ 1 \leq i \leq q - t\}$ a set of chosen points of $(l\backslash l_{\infty})(\mathbb{F}_q) \cap \mathcal{X}.$ So $\mu \neq 0$ and $\beta_i \neq 0$ for $i = 1, ..., q - t$. Hence,
        	$$\left(\begin{array}{ccccc}
        		&& \vdots&& \\
        		1 & \beta_i & \beta_i^2 & \cdots & \beta_i^{q - t - 1} \\
        		&& \vdots &&
        	\end{array}  \right)_{i = 1, ..., q - t} \left(\begin{array}{c}
        		\mu g_{q - 3}(1, v) - a_{00} \\
        		\mu g_{q - 4}(1, v) \\
        		\vdots \\
        		\mu g_{t - 2}(1, v)
        	\end{array}\right) = \left(\begin{array}{c}
        		0 \\
        		0 \\
        		\vdots \\
        		0
        	\end{array}\right).$$
        	This implies that $\mu g_{q - 3}(1, \mu) - a_{00} = \mu g_{q - 4}(1, \mu) = \mu g_{t - 2}(1, \mu) = 0$. So $F(1, \mu, \beta) = 0$ for any $\beta \in \mathbb{F}_q^{*}$. Therefore, $l(\mathbb{F}_q) \backslash (\{Q_0\} \cup l_{\infty}(\mathbb{F}_q)) \subseteq \mathcal{X}(\mathbb{F}_q)$.
    
    \end{proof}
    
    \begin{corollary}\label{CorRQ2}
        Suppose there is a point $Q \in Z(\mathcal{X})$ such that $r = \psi_{q - 1}(Q) \geq 2$. If these $r$ lines are $l_1, ..., l_r$ then $\{ (l_i(\mathbb{F}_q) \backslash \{Q\}) \cap Z(\mathcal{X})\ |\ i = 1, ..., r\}$ is contained in a line.
    \end{corollary}
    \begin{proof}
        Let $Q_i = (l_i(\mathbb{F}_q) \backslash \{Q\}) \cap Z(\mathcal{X})$ with $i = 1, ..., r$. Since $r \geq 2$, we can consider the line $l_{\infty} := \overline{Q_1Q_2}$. Since $l_i \in \mathcal{A}_{q - 1}$, then $\# (l_i \backslash l_{\infty}) \cap \mathcal{X}(\mathbb{F}_q) \geq (q + 1) - 3 = q - 2$ for $i = 3, ..., r$. By Proposition \ref{PropQ}, we have $l_i \backslash (l_{\infty} \cup \{Q\}) \subseteq \mathcal{X}(\mathbb{F}_q)$. Therefore, 
        $$\{ (l_i(\mathbb{F}_q) \backslash \{Q\}) \cap Z(\mathcal{X})\ |\ i = 1, ..., r\} \subseteq l_{\infty}.$$            
    \end{proof}
  
  We now prove some interesting Corollaries to Proposition \ref{PropQ}. They can be thought of as partial negative answers to Question 1 when assuming stronger conditions on the structure of  $Z(\mathcal{X})$. %that give a negative answer to Question 1  under some progressively looser conditions on the structure of $Z(\mathcal{X})$ when $q$ is big enough.} 
    \begin{corollary}\label{Cor1}
        If there are two lines $l_{\infty}^1, l_{\infty}^2 \in \check{\mathbb{P}}^2(\mathbb{F}_q)$ such that $(l_{\infty}^1 \cup l_{\infty}^2)(\mathbb{F}_q) \subseteq Z(\mathcal{X})$, then $a_0 = 3$.
    \end{corollary}
    \begin{proof}
        Let $Z^*(\mathcal{X}) := Z(\mathcal{X}) \backslash (l_{\infty}^1 \cup l_{\infty}^2)$ and $Q \in l_{\infty}^1 \cap l_{\infty}^2$. Since $\#Z(\mathcal{X}) = 3q$, we have $\#Z^*(\mathcal{X}) = q - 1$. Since $\#\check{Q}(\mathbb{F}_q) \backslash \{l_{\infty}^1, l_{\infty}^2\} = q - 1$ and $\text{deg}(\mathcal{X}) = q - 1$, then each line $l \in \check{Q}(\mathbb{F}_q) \backslash \{l_{\infty}^1, l_{\infty}^2\}$ contains exactly one point of $Z^*(\mathcal{X})$. By Corollary \ref{CorRQ2}, $Z^*(\mathcal{X})$ is contained in a line. Therefore, $a_0 = 3$.
    \end{proof}
    \begin{corollary}\label{CorR4}
        Let $q \geq 8$. If $Q \in Z(\mathcal{X})$ is a point such that $r = \psi_{q - 1}(Q) \geq 4$, then $r = q - 1$. In particular, $a_0 = 3$.
    \end{corollary}
    \begin{proof}
        Since $r \geq 4$, by Lemma \ref{PropQ} and Corollary \ref{CorRQ2}, the lines in $\check{Q}(\mathbb{F}_q)\backslash \mathcal{A}_{q - 1}$ have at most $q - r$ points of $\mathcal{X}(\mathbb{F}_q)$. Then
        \begin{align*}
            q^2 - 2q + 1 = (q - 1)^2 & \leq r(q - 1) + (q + 1 - r)(q - r) \\
            & = qr - r + q^2 + q - qr - qr - r + r^2 \\
            & = q^2 + q - qr + r(r - 2).
        \end{align*}
        Hence $q(r - 3) \leq r(r - 2) - 1$. Since $q \geq 8$, we must have $r \geq 7$. So
        $$q \leq \dfrac{r(r - 2) - 1}{r - 3} = \dfrac{(r + 1)(r - 3) + 2}{r - 3} = r + 1 + \dfrac{2}{r - 3}.$$
        Since $r \leq (q - 1)$, this implies that $r = q - 1$. Therefore, $Q$ is of type $(q - 1)^{q - 1}0^2$ and, by Corollary \ref{Cor1}, $a_0 = 3$.
    \end{proof}
    \begin{corollary}\label{q7a03}
        Let $q \geq 7$. If there is a line $l_1^{\infty} \in \check{\mathbb{P}}^2(\mathbb{F}_q)$ such that $l_1^{\infty}(\mathbb{F}_q) \subseteq Z(\mathcal{X})$, then $a_0 = 3$.
    \end{corollary}
    \begin{proof}
        By Lemma \ref{LemmaPR3}, $a_{q - 1} \geq 3(q - 1)$. Since $q \geq 7$, if  $\psi_{q - 1}(Q) \leq 2$ for every point $Q \in l_1^{\infty}$ then $a_{q - 1} \leq 2(q + 1) < 3(q - 1)$, a contradiction. Hence, there is a point $Q \in l_1^{\infty}$ such that $r = \psi_{q - 1}(Q) \geq 3$. Also, by Lemma \ref{PropQ} and Corollary \ref{CorRQ2}, the other lines have at most $q - r$ points of $\mathcal{X}(\mathbb{F}_q)$. This implies that
        \begin{align*}
            q^2 - 2q + 1 = (q - 1)^2 & \leq 0 + r(q - 1) + (q - r)(q - r)  \\
            & = qr - r + q^2 - 2qr + r^2 \\
            & = q^2 - qr + r(r - 1).
        \end{align*}
        Hence $q(r - 2) \leq r(r - 1) - 1$. Since $q \geq 7$, we must have $r \geq 6$. So
    	$$q \leq \dfrac{r(r - 1) - 1}{r - 2} = \dfrac{(r + 1)(r - 2) + 1}{r - 2} = r + 1 + \dfrac{1}{r - 2}.$$
    	Since $r \leq q - 1$, this implies that $r = q - 1$. Therefore, $Q$ is of type $(q - 1)^{q - 1}0^2$. Then the result follows from Corollary \ref{Cor1}.
        \end{proof}
    	 Let $k_0 = \min\{i\ |\ a_i \neq 0\}$. By the previous Corollary, in order to prove our main result for $q \geq 7$, it is enough to show that $k_0 = 0$.  We start by giving an upper bound for  $k_0$.
    \begin{lem}\label{Lemmak_0Q4}
      We have that   $k_0 \leq q - 4$.
    \end{lem}
    \begin{proof}
        Suppose that $k_0 \geq q - 3$. Hence, by Lemma \ref{PropA_i}, we have
        \begin{itemize}
            \item[(1)] $a_{q - 3} + a_{q - 2} + a_{q - 1} = q^2 + q + 1$;
            \item[(2)] $(q - 3)a_{q - 3} + (q - 2)a_{q - 2} + (q - 1)a_{q - 1} = (q + 1)(q - 1)^2$;
            \item[(3)] $(q - 3)(q - 4)a_{q - 3} + (q - 2)(q - 3)a_{q - 2} + (q - 1)(q - 2)a_{q - 1} = q(q - 2)(q - 1)^2$.
    	\end{itemize}
        A direct computation shows that the system above implies that $2a_{q - 3} = 3(q^2 - 3q + 2)$, $a_{q - 2} = -2(q^2 - 5q + 4)$ and $2a_{q - 1} = 3(q^2 - 3q + 4)$. Since $q \ge 5$, then 
        $$a_{q - 2} = -2(q^2 - 5q + 4) < 0,$$
        a contradiction. Therefore, $k_0 \leq q - 4$.
    \end{proof}
    We now prove a lower bound for $k_0$ whenever $k_0 \neq 0$. We start with the following lemma:
    \begin{lem}\label{Lemaik0}
        $\sum_{i = k_0}^{q - 1} (i - k_0)(i - q + 2)a_i = 3(q - 1)^2 - 3k_0$. In particular, 
        $$(q - k_0 - 1)a_{q - 1} \geq 3(q - 1)^2 - 3k_0.$$
    \end{lem}
    \begin{proof}
        First, note that
        \begin{align*}
            (i - k_0)(i - q + 2) & = i(i - 1) + i(3 - q) + k_0(q - 2 - i) \\
            & = i(i - 1) + i(3 - q) -i k_0 + k_0(q - 2) \\
            & = i(i - 1) + i(3 - q - k_0) + k_0(q - 2).
        \end{align*}
        By Lemma \ref{PropA_i}, we have 
        \begin{align*}
            \dsum_{i = k_0}^{q - 1} (i - k_0)(i - q + 2)a_i & = \dsum_{i = k_0}^{q - 1} i(i - 1)a_i + (3 - q - k_0)\dsum_{i = k_0}^{q - 1} ia_i + k_0(q - 2)\dsum_{i = k_0}^{q - 1} a_i \\
        	& = q(q - 2)(q - 1)^2 + (3 - q - k_0)(q + 1)(q - 1)^2 + \\
                &\ \ \ \ \ + k_0(q - 2)(q^2 + q + 1) \\
                & = 3(q - 1)^2 - 3k_0.
        \end{align*}
        Moreover, if $i = q - 1$ then $(i - k_0)(i - q + 2) = (q - k_0 - 1)$ and if $k_0 \leq i \leq q - 2$ then $(i - k_0)(i - q + 2) \leq 0$. Also, by Lemma \ref{Lemmak_0Q4}, $k_0 \leq q - 4$. Hence, $3(q - 1)^3 - 3k_0 \geq 0$. Therefore, $(q - k_0 - 1)a_{q - 1} \geq 3(q - 1)^2 - 3k_0$.
    \end{proof}
    \begin{prop}\label{q7k0neq1}
    Let $q \geq 7$. If $k_0 \neq 0$, then $k_0 \geq 2$.
    \end{prop}
    \begin{proof}
        Suppose that $k_0 = 1$. Let $l_1 \in \mathcal{A}_1$. Also, consider $P_0 = l_1 \cap \mathcal{X}(\mathbb{F}_q)$ and 
        $$l_1 \cap Z(\mathcal{X}) = \{Q_1, ..., Q_q\}.$$
        By Corollary \ref{Lema3}, $P_0$ is of type $(q - 1)^q1^1$. If $r_i := \psi_{q - 1}(Q_i)$ for $i = 1, ..., q$ then 
        $$r_i(q - 1) + 1 \leq \text{N}_q(\mathcal{X}) = (q - 1)^2$$
        Hence, $r_i \leq q - 2$. If $r_i \geq 3$, by Lemma \ref{PropQ} and Corollary \ref{CorRQ2}, the lines in $\check{Q}_i(\mathbb{F}_q) \backslash \mathcal{A}_{q - 1}$ have at most $q - r_i$ points in $\mathcal{X}(\mathbb{F}_q)$. On the other hand, we have
        \begin{align*}
            q^2 - 2q + 1 = (q - 1)^2 & \leq 1 + r_i(q - 1) + (q - r_i)(q - r_i) \\
            & = 1 + r_iq - r_i + q^2 - 2qr_i + r_i^2 \\
        	& = q^2 + 1 - qr_i + r_i(r_i - 1),
        \end{align*}
        hence, $q (r_i - 2) \leq (r_i - 1)r_i.$ Since $q \geq 7$, we must have $r_i \geq 6$. So
        $$q \leq \dfrac{(r_i - 1)r_i}{(r_i - 2)} = \dfrac{(r_i + 1)(r_i - 2) + 2}{r_i - 2} = r_i + 1 + \dfrac{2}{r_i - 2}.$$
        This implies that $r_i \geq q - 1$, a contradiction. Therefore $r_i \leq 2$ for $i = 1, ..., q$, hence, $a_{q - 1} \leq q + 2q = 3q$. By Lemma \ref{Lemaik0}, $(q - 2)a_{q - 1} \geq 3(q - 1)^2 - 3 = 3q(q - 2).$ So $a_{q-1} = 3q$. By Lemma \ref{Lemaik0}, we have
        $$\sum_{i = 1}^{q - 1} (i - 1)(i - q + 2)a_i = 3q(q - 2).$$
        Since $(q - 2)a_{q - 1} = 3q(q - 2)$, this implies that $a_2 = \cdots = a_{q - 3} = 0$. By Lemma \ref{PropA_i}, we have
        \begin{itemize}
            \item[(1)] $a_1 + a_{q - 2} + 3q = q^2 + q + 1$,
        	\item[(2)] $a_1 + (q - 2)a_{q - 2} + 3q(q - 1) = (q + 1)(q - 1)^2.$ 
        \end{itemize}
        A direct computation shows that the system above implies that $(q - 3)a_1 = 3(q - 1)$ and $(q - 3)a_{q - 2} = q(q^2 - 5q + 4)$. Note that
        $$a_1 = \dfrac{3(q - 1)}{q - 3} = 3 + \dfrac{6}{q - 3}$$
        Since $q \geq 7$ then $(q - 3) \nmid 6$, but $a_1$ is a integer, a contradiction. Therefore, $k_0 \geq 2$.
    \end{proof}

\section{Characterization of optimal Sziklai curves of degree $q - 1$}

        In this section, we provide the characterization of optimal Sziklai curves of degree $q-1$. First, we deal with the cases $q =5,7$, as they need  some \emph{ad hoc} techniques.

        \begin{prop}\label{q=5}
            If $\mathcal{X} \in C_4(\mathbb{F}_5)$ and $\text{N}_5(\mathcal{X}) = 16$, then there exist $\alpha, \gamma, \beta \in \mathbb{F}_5^{*}$ with $\alpha + \beta + \gamma = 0$ such that $\mathcal{X} \simeq_{{\rm proj}} \mathcal{C}_{(\alpha, \beta, \gamma)}$ over $\mathbb{F}_5$.
        \end{prop}
        \begin{proof}
            By Lemma \ref{Lemmak_0Q4}, we have $k_0 \leq 1$. Suppose that $k_0 = 1$. Let $l_1 \in \mathcal{A}_1$ with $l_1 \cap \mathcal{X}(\mathbb{F}_q) = \{P_0\}$. By Corollary \ref{Lema3}, $P_0$ is of type $4^51^1$. In this case, $l_1$ is the tangent line to $\mathcal{X}$ at $P_0$. Since $\text{deg}(\mathcal{X}) = 4$, then the divisor 
            $$l_1 \cdot \mathcal{X} := \sum \text{I}(P, l_1 \cap \mathcal{X}) P =  2P_0 + P_1 + P_2$$
            for some points $P_1, P_2 \in l_1 \cap \mathcal{X}$, where $\text{I}(P, l_1 \cap \mathcal{X})$ is the intersection multiplicity of $l_1$ and $\mathcal{X}$ at $P$. Since the divisor $l_1 \cdot \mathcal{X}$ is defined over $\mathbb{F}_q$, applying the $5$-Frobenius map $\Psi_5$, we have $P_1 + P_2 + 2P_0 = \Psi_5(P_1) + \Psi_5(P_2) + 2P_0$, which implies that $P_1, P_2 \in \mathbb{P}^2(\mathbb{F}_{25})$. If $P_1 = P_2$ then $P_1 = P_2 = P_0$ and $P_0$ is a inflexion point. On the other hand, by Stohr-Voloch's Theorem \ref{SVolochTheo}, we have
            $$32 = 2\text{N}_q(\mathcal{X}) \leq 32 - \sum_{P \in \mathcal{X}} A(P).$$
            This implies that $0 = A(P_0) = j_2(P_0) - 2$, hence, $\text{I}(P_0, l_1 \cap \mathcal{X}) = j_2(P_0) = 2,$ a contradiction. So $P_1 \neq P_2$ and $P_1, P_2 \in \mathbb{P}^2(\mathbb{F}_{25}) \backslash \mathbb{P}^2(\mathbb{F}_{5})$. By Lemma \ref{PropA_i}, we have
            \begin{itemize}
            	\item[(i)] $a_1 + a_2 + a_3 + a_4 = 31;$
            	\item[(ii)] $a_1 + 2a_2 + 3a_3 + 4a_4 = 96;$
            	\item[(iii)] $2a_2 + 6a_3 + 12a_4 = 240.$
            \end{itemize}
            A direct computation shows that the system above implies that $a_1 = 21 - a_4$, $a_2 = 3(a_4 - 15)$ and $a_3 = 55 - 3a_4$. Since $a_1, a_2, a_3 \geq 0$, then $15 \leq a_4 \leq 18$. Hence $a_1 \geq 3$. Follows from the above that $a_1 \geq 3$ implies that $\text{N}_{25}(\mathcal{X}) \geq 16 + 3*2 = 22$. On the other hand, let $L(t) = \prod_{i = 1}^6(1 - \omega_it)$ the $L$-polynomial of $\mathcal{X}$ into linear factors in some finite extension of $\mathbb{Q}$. By \cite[Theorem 9.10]{Torres}, we have
            $$\text{N}_{q^n}(\mathcal{X}) = q^n + 1 - \sum_{i = 1}^{2g} \omega_i^n.$$
            Since $16 = \text{N}_5(\mathcal{X}) = 5 + 1 + 3 \cdot 4 - 2 = 5 + 3 \cdot 4 - 1$. As mentioned in \cite[fact 3.3]{Lauter}, the curve $\mathcal{X}$ has zeta function of type $[4, 4, 2]$, this means that
            $$\omega_1 + \overline{\omega_1} = -4,\ \omega_2 + \overline{\omega_2} = -4 \text{ and } \omega_3 + \overline{\omega_3} = -2.$$
            Hence, $\sum_{i = 1}^6 \omega_i^2 = 36 - 2(|\omega_1|^2 + |\omega_2|^2 + |\omega_3|^2).$
            By \cite[Theorem 9.16]{Torres}, $|\omega_i|^2 = 5$. This implies that
            $$22 \leq \text{N}_{25}(\mathcal{X}) = 26 - \sum_{i = 1}^{6} \omega_i^2 = 26 - (36 - 30) = 20,$$
            a contradiction. Then $k_0 = 0$. 
            
            Suppose that $a_1 \neq 0$. Let $l_0 \in \mathcal{A}_0$ and $l_1 \in \mathcal{A}_1$. By Corollary \ref{Lema3}, we have 
            $$Q = l_0 \cap l_1 \in Z(\mathcal{X}).$$
            If $r_Q = \psi_4(Q)$, by Lemma \ref{LemmaType}, we have $16 \leq 0 + 1 + 4r + 3(4 - r) = 13 + r,$ hence, $r \geq 3$. This means that $Q$ is of type $4^33^11^10^1$. On the other hand, since $\psi_4(Q) \geq 3$, by Proposition \ref{PropQ}, we must have $\psi_{3}(Q) = 0$, a contradiction. So $a_1 = 0$. By Lemma \ref{PropA_i}, we have
            \begin{itemize}
            	\item[(1)] $a_0 + a_2 + a_3 + a_4 = 31;$
            	\item[(2)] $2a_2 + 3a_3 + 4a_4 = 96;$ 
            	\item[(3)] $2a_2 + 6a_3 + 12a_4 = 240.$
            \end{itemize}
            A direct computation shows that the system above implies that $3a_0 = 21 - a_4$, $a_2 = 2(a_4 - 12)$ and $3a_3 = 144 - 8a_4$. If $a_0 = 1$, then $a_2 = 12$ and $a_3 = 0$. Let $l_2 \in \mathcal{A}_2$ and $P_0 \in l_2 \cap \mathcal{X}(\mathbb{F}_5)$. If $r = \psi_4(P_0)$, by Lemma \ref{LemmaType}, we have $16 \leq 2 + 3r + 2(5 - r) = 12 + r.$ Since $2 + 3r \leq 16$, this implies that $r = 4$. Hence, $\psi_2(P_0) = 1$, $\psi_3(P_0) = 1$ and $\psi_4(P_0) = 4$, a contradiction. Therefore, $a_0 \geq 2$ and the result follows Corollary \ref{Cor1}.
        \end{proof}
        
        \begin{prop}\label{q=7}
            If $\mathcal{X} \in C_6(\mathbb{F}_7)$ and $\text{N}_7(\mathcal{X}) = 36$, then there exist $\alpha, \gamma, \beta \in \mathbb{F}_7^{*}$ with $\alpha + \beta + \gamma = 0$ such that $\mathcal{X} \simeq_{{\rm proj}} \mathcal{C}_{(\alpha, \beta, \gamma)}$ over $\mathbb{F}_7$.
        \end{prop}
        \begin{proof}
            First, recall the definition of  $k_0 = \min\{i\ |\ a_i \neq 0\}$. Also, by Proposition \ref{q7k0neq1} and Corollary \ref{q7a03}, it is enough to prove that $k_0 \leq 1$. 
            
            By Lemma \ref{LemaArc}, we have that $\mathcal{X}(\mathbb{F}_7) \subseteq \mathbb{P}^2(\mathbb{F}_7)$ is a $(36, 6)$-arc. Up to projective equivalence, there are exactly 194 $(36, 6)$-arcs in $\mathbb{P}^2(\mathbb{F}_7)$, see \cite[Remark 1]{ArcsPG27}. The full list can be found online at \url{http://mars39.lomo.jp/opu/36_3_30.txt}, where the points of such arcs are arranged as a generator matrix for a $[36, 3, 30]_7$-code together with the weight enumerator. Recall that, by the relation between projective $[n, 3, d]_q$-codes and $(n, n - d)$-arcs in $\mathbb{P}^2(\mathbb{F}_q)$, we have
            $$6 \cdot (a_0, a_1, a_2, a_3, a_4, a_5, a_6) = (c_{36}, c_{35}, c_{34}, c_{33}, c_{32}, c_{31}, c_{30}).$$
            By Lemma \ref{Lemmak_0Q4}, $k_0 \leq 3$. Hence, from our initial observation, it is enough to prove that the condition $k_0 \in \{2, 3\}$ leads to a contradiction, equivalently, the exactly six $(36, 6)$-arcs of the $[36, 3, 30]_7$-codes such that the weight of a codeword is at most $34$ is not projectively equivalent to $\mathcal{X}(\mathbb{F}_7)$.
            
            Suppose that $k_0 = 3$. Since there is only one $[36, 3, 30]_7$-code such that the maximal weight of a codeword is exactly 33, then there is only one $(36, 6)$-arc with $k_0 = 3$, up to projective equivalence. So we can choose coordinates $X, Y, Z$ of $\mathbb{P}^2$ such that
            \begin{align*}
                \mathcal{X}(\mathbb{F}_7) = \{ & (1:1:3), (1:1:5), (1:2:3), (1:2:5), (1:2:6), (1:3:3), \\
                & (1:3:4), (1:3:5), (1:3:6), (1:3:0), (1:4:2), (1:4:4), \\
                & (1:4:5), (1:4:6), (1:4:0), (1:5:2), (1:5:3), (1:5:4), \\
                & (1:5:6), (1:5:0), (1:6:1), (1:6:3), (1:6:4), (1:6:5), \\
                & (1:6:0), (1:0:1), (1:0:2), (1:0:4), (1:0:6), (1:0:0), \\ 
                & (0:1:3), (0:1:4), (0:1:5), (0:1:6), (0:1:0), (0:0:1) \}.
            \end{align*}
            Let
            $$G = YZ(Z - 3Y)(Z - 4Y)(Z - 5Y)(Z - 6Y) \in \mathbb{F}_7[X, Y, Z]$$
            $$H = X(X + Y - Z)(2X + Y - Z)(X + 2Y - 2Z) \in \mathbb{F}_7[X, Y, Z]$$
            A direct computation shows that $\#(\VarP(G) \cap \VarP(H)) = 24$ and $ \VarP(G) \cap \VarP(H) \subseteq \mathcal{X}(\mathbb{F}_7)$. Let $F \in \mathbb{F}_7[X, Y, Z]$ be a homogeneous equation for $\mathcal{X}$ over $\mathbb{F}_7$, by Noether's “AF + BG” Theorem \ref{Noether}, we have
            $$F = G + (\alpha_1X^2 + \alpha_2Y^2 + \alpha_3Z^2 + \alpha_4XY + \alpha_5XZ + \alpha_6YZ) \cdot H$$
            for some $\alpha_1, ..., \alpha_6 \in \mathbb{F}_7$. Since $0 = F(1, 0, 0) = 2\alpha_1$,
            \begin{align*}
                0  = F(1, 5, 4) & = G(1, 5, 4) + (5^2\alpha_2 + 4^2\alpha_3 + 5 \alpha_4 + 4 \alpha_5 4 + 20 \alpha_6)H(1, 5, 4) \\
                & = 4(4\alpha_2 + 2\alpha_3 + 5 \alpha_4 + 4 \alpha_5 + 6 \alpha_6), \\
                0  = F(1, 4, 0) & = G(1, 4, 0) + (4^2\alpha_2 + 4\alpha_4)H(1, 4, 0) \\
                & = 4(2\alpha_2 + 4\alpha_4),\\
                0  = F(1, 0, 6) & = G(1, 0, 6) + (6^2\alpha_3 + 6\alpha_5)H(1, 0, 6) \\
                & = 4(\alpha_3 + 6\alpha_5), \\
                0  = F(1, 2, 5) & = G(1, 2, 5) + (2^2\alpha_2 + 5^2\alpha_3 + 2\alpha_4 + 5\alpha_5 + 10\alpha_6)H(1, 2, 5) \\
                & = 4(4\alpha_2 + 4\alpha_3 + 2\alpha_4 + 5\alpha_5 + 3\alpha_6), \\
                0  = F(1, 6, 4) & = G(1, 6, 4) + (6^2\alpha_2 + 4^2\alpha_3 + 6\alpha_4 + 4\alpha_5 + 24\alpha_6YZ)H(1, 6, 4) \\
                & = 4(\alpha_2 + 2\alpha_3 + 6\alpha_4 + 4\alpha_5 + 3\alpha_6YZ),
            \end{align*}
            It is easy to check that this implies that $\alpha_1 = \cdots = \alpha_6 = 0$. Therefore, $F = G$ and $\mathcal{X}$ has $\mathbb{F}_7$-linear components, a contradiction. 
            
            Suppose that $k_0 = 2$. Since there are five $[36, 3, 30]_7$-codes such that the maximal weight of a codeword is exactly 34, then there are five $(36, 6)$-arc with $k_0 = 2$, up to projective equivalence. For instance, we can choose coordinates $X, Y, Z$ of $\mathbb{P}^2$ such that
            \begin{align*}
                \mathcal{X}(\mathbb{F}_7) = \{ & (1:1:4), (1:1:5), (1:1:6), (1:1:0), (1:2:2), (1:2:4), \\
                & (1:2:6), (1:2:0), (1:3:2), (1:3:3), (1:3:4), (1:3:5), \\
                & (1:4:2), (1:4:3), (1:4:5), (1:4:6), (1:5:3), (1:5:5), \\
                & (1:5:6), (1:5:0), (1:6:1), (1:6:2), (1:6:4), (1:6:6), \\
                & (1:6:0), (1:0:2), (1:0:3), (1:0:4), (1:0:5), (1:0:0), \\ 
                & (0:1:3), (0:1:4), (0:1:5), (0:1:6), (0:1:0), (0:0:1) \}.
            \end{align*}
            Let $Q_0 = (1: 5: 2) \in Z(\mathcal{X})$. A direct computation shows that 
            $$\check{Q_0}(\mathbb{F}_7) \cap \mathcal{A}_5 = \{ x + 3y + 6z, x + 4y\} \text{ and } \check{Q_0}(\mathbb{F}_7) \cap \mathcal{A}_6 = \{ y + z, 4x +3y + z, x + 3z \}.$$
            Since $\psi_6(Q_0) = 3$, by Corollary \ref{CorRQ2} and Proposition \ref{PropQ}, we must have $\psi_5(Q_0)  = 0$, a contradiction. 
            
            The other four $(36, 6)$-arc with $k_0 = 2$ can be dealt with by using a similar argument. In each of these cases, we have a contradiction.
            Therefore, none of these $(36, 6)$-arcs can be equal to $\mathcal{X}(\mathbb{F}_7)$, and our assertion follows. 
        \end{proof}
        
        \begin{rem}
            As a byproduct of Proposition \ref{q=7}, we have that, of the 194 nonequivalent $(36,6)$-arcs in $\mathbb{P}^2(\mathbb{F}_7)$, only one is obtained as the set of rational points of an irreducible plane curve of degree $6$. 
        \end{rem}
        
        We are now in a position to prove our main result. 
        
        \begin{theorem}\label{MainResult}
             Let $\mathcal{X} \in C_{q - 1}(\mathbb{F}_q)$. If $\text{N}_q(\mathcal{X}) = (q - 1)^2$ and $q \geq 5$, then there exist $\alpha, \gamma, \beta \in \mathbb{F}_q^{*}$ with $\alpha + \beta + \gamma = 0$ such that $\mathcal{X} \simeq_{{\rm proj}} \mathcal{C}_{(\alpha, \beta, \gamma)}$ over $\mathbb{F}_q$.
        \end{theorem}
        \begin{proof}
            Let $k_0 = \min\{i\ |\ a_i \neq 0\}.$ If $q \ge 8$, then by Proposition \ref{q7k0neq1} and Corollary \ref{CorR4}, it is enough to prove that either $k_0 \leq 1$ or there is a point $Q \in Z(\mathcal{X})$ such that $\psi_{q - 1}(Q) \geq 4$. 
            
            By way of contradiction, assume that $k_0 \geq 2$ and $r_Q := \#\check{Q}(\mathbb{F}_q) \cap \mathcal{A}_{q - 1} \leq 3$ for every point $Q \in Z(\mathcal{X})$. Since $\#Z(\mathcal{X}) = 3q$, the latter hypothesis implies that $2a_{q - 1} \leq 9q$.
            
            Now, assume that $q \geq 11$. If for every point $P \in \mathcal{X}(\mathbb{F}_q)$ we have $\psi_{q - 1}(P) \geq 5$, then $a_{q - 1} \geq 5(q - 1) > 9q/2$, a contradiction. Let $P_0 \in \mathcal{X}(\mathbb{F}_q)$ such that $r_0 = \psi_{q - 1}(P) \leq 4$. By Lemma \ref{LemmaPR3}, we have that $r_0 \in \{3, 4\}$. We distinguish two cases, namely  $r_0 = 3$ or $r_0 = 4$. 
            
            If $r_0 = 3$, then by Lemma \ref{LemmaPR3}, $P_0$ is of type $(q - 1)^3(q - 2)^{q - 2}$. Let $l \in \check{P}(\mathbb{F}_q) \cap \mathcal{A}_{q - 2}$ and $Q \in l \cap Z(\mathcal{X})$. If $r_Q := \psi_{q - 1}(Q) = 3$, by Lemma \ref{PropQ} and Corollary \ref{CorRQ2}, then $l \in \mathcal{A}_{q - 1}$, a contradiction. So there are at least $3(q - 2)$ points in $Z(\mathcal{X})$ such that $r_Q \leq 2$. Therefore, $2a_{q - 1} \leq 2(3(q - 2)) + 3*6$. So $a_{q - 1} \leq 3(q - 2) + 9 = 3(q + 1)$. By Lemma \ref{Lemmak_0Q4}, $k_0 \leq q - 4$ and, by Lemma \ref{Lemaik0}, we have
                $$\dfrac{3(q - 1)^2 - 3k_0}{q - k_0 - 1} \leq a_{q - 1} \leq 3(q + 1).$$
            This implies that $3q(k_0 - 2) + 6 = 3(q - 1)^2 - 3k_0 - (q - k_0 - 1)(3q + 3) \leq 0$. Therefore, $k_0 \leq 2 - 2/q < 2,$ a contradiction. 
            
            Now, suppose that $r_0 = 4$. Let $s = \psi_{q - 2}(P)$. We have
            \begin{align*}
                (q - 1)^2 & \leq 1 + 4(q - 2) + s(q - 3) + (q - 3 - s)(q - 4) \\
                & = 4q - 7 + sq - 3s + q^2 - 3q - sq - 4q + 12 + 4s \\
                & = (q - 1)^2 + 4 + s - q.
            \end{align*}
            This implies that $s \geq q - 4$, hence $P$ is of type $(q - 1)^4(q - 2)^{q - 4}(q - 5)^1.$ As in the previous case, there are at least $3(q - 4)$ points in $Z(\mathcal{X})$ such that the image by $\psi_{q - 1}$ is less than or equal to $2$. Therefore, $2a_{q - 1} \leq 2(3(q - 4)) + 3 \cdot 12$, hence, $a_{q - 1} \leq 3(q - 4) + 18 = 3(q + 2)$. Since $q \geq 11$, there is point $P \in \mathcal{X}(\mathbb{F}_q)$ such that $\psi_{q - 1}(P) = 3$, otherwise, $a_{q - 1} \geq 4(q - 1) > 3(q + 2).$ Again, this implies that $k_0 < 2$, a contradiction.
                
            We are then left with the cases $q = 5,7, 8, 9$. As the cases $q=5,7$ have already been dealt with in Propositions \ref{q=5} and \ref{q=7} respectively, we only need to consider the cases  $q =8,9$.  We proceed with a careful case-by-case analysis.  
             
            For $q = 9$, by Lemma \ref{Lemaik0}, we have
            \begin{align}\label{a8}
            	\dfrac{3(64 - k_0)}{8 - k_0} \leq a_8 \leq 40.
            \end{align}
            By Lemma \ref{Lemmak_0Q4}, we have $k_0 \leq 5$, and  then $37k_0 - 128 = 3(64 - k_0) - 40(8 - k_0) \leq 0$. This in turn implies $k_0 \in \{ 2, 3 \}$. 
            
            Now, we prove that the condition $k_0 = 2$ leads to a contradiction for $q = 9$. If $k_0 = 2$, by the inequality (\ref{a8}), $a_8 \geq 31$. Let $l_2 \in \mathcal{A}_2$ with $l_2 \cap \mathcal{X}(\mathbb{F}_9) = \{ P_0, P_1\}$ and $l_2 \cap Z(\mathcal{X}) = \{ P_2, ..., P_9\}.$ Let $r_j := \psi_{8}(P_j)$ for $j = 0, 1, 2, ..., 9$. By Lemma \ref{LemmaType}, if $j \in \{0, 1\}$, then $64 \leq 2 + 7r_j + 6(9 - r_j) = 56 + r_j.$ Since $2 + 7r_j \leq 64$, this implies that $r_j = 8$ for $j = 0, 1$. If $r_j = 3$ for some $j \in \{2, ..., 9\}$, follows from Lemma \ref{LemmaType} and Proposition \ref{PropQ} that the other lines in $\check{Q}_j(\mathbb{F}_q)$ contain at most $6$ points of $\mathcal{X}(\mathbb{F}_9)$, hence, $64 \leq  2 + 3\cdot8 + 6\cdot6 = 62,$ a contradiction. Then $r_j \leq 2$ when $j = 2, ..., 9$. This implies that $a_8 \leq 2\cdot8 + 8\cdot 2 = 32$. If $r_P = \psi_8(P) \geq 4$ for every point $P \in \mathcal{X}(\mathbb{F}_q) \backslash \{P_1, P_2\}$, then
            $$a_8 \geq \dfrac{8\cdot2 + 4((q - 1)^2 - 2)}{q - 1} = 33,$$
            a contradiction. Then there is a point $P \in \mathcal{X}(\mathbb{F}_q)$ such that $r_P = 3$. Again, this implies that $a_{8} \leq 3(q + 1) = 30$, a contradiction.
            
            We now prove that the condition $k_0 = 3$ leads to a contradiction for $q = 9$. If $k_0 = 3$, by inequality (\ref{a8}), $a_8 \geq 37$. Let $l_3 \in \mathcal{A}_3$ with $l_3 \cap \mathcal{X}(\mathbb{F}_9) = \{ P_0, P_1, P_2\}$ and $l_3 \cap Z(\mathcal{X}) = \{ P_3, ..., P_9\}.$ Let $r_j := \psi_{8}(P_j)$ where $i = 0, 1, 2, ..., 9$. By Lemma \ref{LemmaType} if $j \in \{0, 1, 2\}$, then $64 \leq 3 + 7r_j + 6(9 - r_j) = 57 + r_j.$ Since $3 + 7r_j \leq 64$, this implies that $r_j \in \{7, 8\}$ for $j = 0, 1, 2$. If $r_j = 3$ for some $j \in \{3, ..., 9\}$, it follows from Lemma \ref{LemmaType} and Proposition \ref{PropQ} that the other lines in $\check{Q}_j(\mathbb{F}_q)$ contain at most $6$ points of $\mathcal{X}(\mathbb{F}_9)$, hence, $64 \leq 3 + 3\cdot 8 + 6\cdot 6 = 63,$ a contradiction. Then $r_j \leq 2$ when $j = 3, ..., 9$. This implies that $a_8 \leq 3\cdot8 + 7\cdot2 = 38$. If $r_P = \psi_{8}(P) \geq 5$ for every point $P \in \mathcal{X}(\mathbb{F}_q)$, then $a_8 \geq 5(q - 1) = 40,$ a contradiction.  There is then a point $P \in \mathcal{X}(\mathbb{F}_q)$ such that $\psi_{8}(P) \leq 4$. Again, this implies that $a_8 \leq 3(q + 2) = 33$, a contradiction.
            
            For $q = 8$, by Lemma \ref{Lemaik0}, we have
            \begin{align}\label{a7}
            	\dfrac{3(49 - k_0)}{7 - k_0} \leq a_{7} \leq 36
            \end{align}
            By Lemma \ref{Lemmak_0Q4}, we know that $k_0 \leq 4$. Then $3(11k_0 - 35) = 3(49 - k_0) - 36(7 - k_0) \leq 0.$ Hence, $k_0 \in \{ 2, 3 \}$.
            
            We are then left with the case $q= 8$. As in the previous case, we prove that the condition $k_0 = 2$ leads to a contradiction. If $k_0 = 2$, by the  inequality (\ref{a7}), $a_{7} \geq 29.$
            Let $l_2 \in \mathcal{A}_2$ with $l_2 \cap \mathcal{X}(\mathbb{F}_8) = \{ P_0, P_1\}$ and $l_2 \cap Z(\mathcal{X}) = \{ P_2, ..., P_8\}.$ Let $r_j := \psi_{7}(P_j)$ where $j = 0, 1, 2, ..., 8$. By Lemma \ref{LemmaType} if $j \in \{0, 1\}$, then $49 \leq 2 + 6r_j + 5(8 - r_j) = 42 + r_j.$ Since $2 + 6r_j \leq 49$, then $r_j = 7$ for $i = 0, 1$. If $r_j = 3$ for some $j = 2, ..., 8$, by Lemma \ref{LemmaType} and Proposition \ref{PropQ}, the other lines in $\check{Q}_j(\mathbb{F}_q)$ contain at most $5$ points of $\mathcal{X}(\mathbb{F}_8)$. We thus obtain that  $49 \leq 2 + 3\cdot7 + 5\cdot5 = 48,$ a contradiction. Hence, $r_j \leq 2$ when $j = 2, ..., 8$. Therefore $a_{7} \leq 2\cdot7 + 7\cdot2 = 28$, a contradiction.
            
            Finally, we prove that the condition $k_0 = 3$ leads to a contradiction for $q = 8$. If $k_0 = 3$, by the inequality (\ref{a7}), $a_7 \geq 35$. Let $l_3 \in \mathcal{A}_3$ with $l_3 \cap \mathcal{X}(\mathbb{F}_8) = \{ P_0, P_1, P_2\}$ and $l_2 \cap Z(\mathcal{X}) = \{ P_3, ..., P_8\}$.  Let $r_j := \psi_{7}(P_j)$ where $i = 0, 1, 2, ..., 8$. By Lemma \ref{LemmaType}, if $j \in \{0, 1, 2\}$ then $49 \leq 3 + 6r_j + 5(8 - r_j) = 43 + r_j.$ Since $3 + 6r_j \leq (q - 1)^2 = 49$, then $r_j \in \{6, 7\}$ for $i = 0, 1, 2$. If $r_j = 3$ for some $j = 2, ..., 8$, by Lemma \ref{LemmaType} and Proposition \ref{PropQ} the other lines in $\check{Q}_j(\mathbb{F}_q)$ contain at most $5$ points of $\mathcal{X}(\mathbb{F}_8)$, then $49 \leq 3 + 3\cdot7 + 5\cdot5 = 49.$ This means that the other $5$ lines in $\check{Q}_j$ are in $\mathcal{A}_{5}$. By Lemma \ref{Lemaik0}, we have $- 2a_4 - 2a_5 + 4a_7 = 138$. This implies that $4a_7 \geq 138 + 2a_5 \geq  138 + 10 = 148$, hence $a_7 \geq 37$, a contradiction. So $r_j \leq 2$ when $j = 2,\ldots, 8$. Therefore, $a_{7} \leq 3\cdot7 + 6\cdot2 = 33$, a contradiction. The proof of our theorem is then completed. 
            
            \end{proof}
            
\section{Final remarks}
\subsection{On $\mathbb{F}_q$-Frobenius classical curves with many points}\label{SectionSV}

Let $\mathcal{C} \subseteq \mathbb{P}^2$ be an irreducible nonsingular algebraic curve of degree $d$ defined over $\mathbb{F}_q$. If $\mathcal{C}$ is $q$-Frobenius classical, by Theorem \ref{babysv}, we have
\begin{equation}\label{SVBound}
    \text{N}_q(\mathcal{C}) \leq \dfrac{d(d + q - 1)}{2}.
\end{equation}
Note that if $d = q - 1$ then $d(d + q - 1)/2 = (q - 1)^2$. This means that the Stöhr-Voloch upper bound for a nonsingular $q$-Frobenius classical plane curve of degree $q - 1$ is equal to Sziklai's upper bound.
Also, by the proof of our main result, it is inferred that the curves attaining the Sziklai bound are $q$-Frobenius classical and have no $\mathbb{F}_q$-rational point of inflection. In other words, Theorem \ref{MainResult} classifies the $\mathbb{F}_q$-Frobenius classical curves of degree $q-1$ attaining the Stöhr-Voloch upper bound (\ref{SVBound}) up to projective equivalence. 
\subsection{On hypersurfaces with many rational points}

In \cite{HK2017}, Homma and Kim gave an upper bound
for the number $\text{N}_q(\mathcal{X})$ of $\mathbb{F}_q$-rational points of a nonsingular hypersurface $\mathcal{X}$ defined over $\mathbb{F}_q$ in an odd-dimensional projective space $\mathbb{P}^n$, namely
\begin{equation}\label{HKodd}
    \text{N}_q(\mathcal{X}) \leq \theta_q\left(m\right)((d - 1)q^{m} + 1), 
\end{equation}
where $m = (n - 1)/2$ and $\theta_q(m) := |\mathbb{P}^m(\mathbb{F}_q)| = q^m + \cdots + q + 1$.
In \cite[Theorem 1.1]{HK2017}, they also characterized, up to projective equivalence, the hypersurfaces attaining the bound (\ref{HKodd}).
%Further, in the same paper, they conjectured the analogue result in an even-dimensional projective space.

In the same paper, they also conjectured the following for the even-dimensional case: if $\mathcal{X} \subseteq \mathbb{P}^n$ is a nonsingular hypersurface defined over $\mathbb{F}_q$ of degree $d$ with $n$ even, then
\begin{equation}\label{HKeven}
\text{N}_q(\mathcal{X}) \leq \Theta_n^{d, q} := \theta_q\left(m - 1\right)((d - 1)q^{m} + 1)    
\end{equation}
where $m = n/2$. 

This conjecture was then proved by Datta in the case $n= 4$ \cite{Datta2009}, and by Tironi \cite{Tironi2022} in the general case. Here, surprisingly, a link with curves that are optimal with respect to the Sziklai bound appears when considering hypersurfaces attaining \eqref{HKeven}.   In fact,  let $\mathcal{X}$ be a hypersurface in $\mathbb{P}^4$  attaining the bound (\ref{HKeven}); then by \cite[Theorem 4.8]{Datta2009} there exists a point $P \in \mathcal{X}(\mathbb{F}_q)$ such that $\mathcal{X} \cap \mathbb{T}_{P}(\mathcal{X})$ is a cone, with center at $P$, over a plane curve $\mathcal{C}$ of degree $d$ defined over $\mathbb{F}_q$ without $\mathbb{F}_q$-linear components and $\text{N}_q(\mathcal{C})$ attains the Sziklai bound. Also, by \cite[Theorem 1]{Tironi2022}, this curve must be nonsingular. For $n \geq 6$, we have an analogous result; by \cite[Theorem 2]{Tironi2022} the bound (\ref{HKeven}) is attained by a  nonsingular hypersurface $\mathcal{X} \subseteq \mathbb{P}^n$ defined over $\mathbb{F}_q$ of degree $d$ with $n \geq 6$ even only if there exists a point $P \in \mathcal{X}(\mathbb{F}_q)$ such that $\mathcal{X} \cap \mathbb{T}_{P}(\mathcal{X})$ is a cone, with center at $P$, over a nonsingular hypersurface $\mathcal{Y} \subseteq \mathbb{P}^{n - 2}$ of degree $d$ defined over $\mathbb{F}_q$ without $\mathbb{F}_q$-linear components and $\text{N}_q(\mathcal{Y}) = \Theta_{n - 2}^{d, q}$. 
Therefore, the extremal hypersurfaces in even dimension can be characterized inductively starting with the ones in $\mathbb{P}^4$, which in turn can be constructed from optimal Sziklai curves. This gives a possible application of Theorem \ref{MainResult}. 

It should be noted, however, that it is still unknown whether all hypersurfaces with a point $P$ such that is a cone, with center at $P$, over a nonsingular hypersurface $\mathcal{Y} \subseteq \mathbb{P}^{n - 2}$ of degree $d$ defined over $\mathbb{F}_q$ without $\mathbb{F}_q$-linear components and $\text{N}_q(\mathcal{Y}) = \Theta_{n - 2}^{d, q}$ attains the bound \eqref{HKeven}.

\end{document}